\documentclass[11pt]{article}
\usepackage{amsmath, amsfonts, amssymb, amsthm}
\newtheorem{Theorem}{Theorem}[section]

\newtheorem{Lemma}[Theorem]{Lemma}

\newtheorem{Corollary}[Theorem]{Corollary}

\newtheorem{Open Problem}[Theorem]{Open Problem}

\newcommand{\mysection}[1]{\section{#1}\setcounter{equation}{0}}

\begin{document}

\title{Existence and asymptotics of normalized solutions  for periodic quasilinear Schr\"{o}dinger system  }

\author{Jianqing Chen\ and Qian Zhang\thanks{Corresponding author:\ qzhang\_fjnu@163.com  (Q. Zhang)}\\
\small  \  College of Mathematics and Informatics \& FJKLMAA, Fujian Normal University, \\
\small  Qishan Campus, Fuzhou 350117, P. R. China }

\date{}

\maketitle

\noindent {\bf Abstract}:  This paper is concerned with a quasilinear Schr\"{o}dinger system. By minimization under a convenient constraint and concentration-compactness lemma, we prove the existence of ground states solution.  Our result covers the case of $\alpha+\beta\in(2,4)$ which seems to be the first result for coupled quasilinear Schr\"{o}dinger system in the periodic situation.

 \medskip



\mysection {Introduction}
Let $N\geq 3$, $u:=u(x),v:=v(x)$ be real valued functions on $\mathbb R^{N}$. In this paper, we  consider the following  quasilinear Schr\"{o}dinger system
\begin{equation}\label{eq11}
\left\{
     \begin{array}{ll}
\aligned &-\Delta u+A(x)u-\frac{1}{2}\triangle(u^{2})u=\frac{2\alpha}{\alpha+\beta}|u|^{\alpha-2}u|v|^{\beta},\\
&-\Delta v+B(x)v-\frac{1}{2}\triangle(v^{2})v=\frac{2\beta}{\alpha+\beta}|u|^{\alpha}|v|^{\beta-2}v,\\
& u(x)\to 0\ \hbox{and}\quad v(x)\to 0\ \hbox{as}\ |x|\to \infty,\endaligned
     \end{array}
   \right.
\end{equation}
where $\alpha,\beta>1$ and $2<\alpha+\beta<\frac{4N}{N-2}$.  In recent years, much attention has been devoted to the quasilinear Schr\"{o}dinger equation of the form
\begin{equation}\label{eq12}
-\Delta u+V(x)u-ku\Delta(u^{2})=|u|^{p-2}u,\ x\in \mathbb R^{N}.
\end{equation}
Such a problem is related  to the existence of standing waves of the following quasilinear Schr\"{o}dinger equation
\begin{equation}\label{eq13}
i\partial_{t}z=-\Delta z+V(x)z-l(|z|^{2})z-k\Delta g(|z|^{2})g'(|z|^{2})z,\ x\in \mathbb R^{N},
\end{equation}
where $V$ is a given potential, $k\in \mathbb R$, $l$ and $g$ are real functions.  (\ref{eq13}) has been used as models in several areas of physics corresponding to various types of $g$. The superfluid film equation in plasma physics has this structure for $g(s) = s $ \cite{k}.  In the case $g(s) = (1 + s)^{\frac{1}{2}}$, Eq. (\ref{eq13}) models the self-channeling of a high-power ultra short laser in matter \cite{r}. Eq. (\ref{eq13}) also appears in fluid mechanics \cite{k, lss}, in the theory of Heidelberg ferromagnetism and magnus \cite{ltz}, in dissipative quantum mechanics and in condensed matter theory \cite{mf}.  When considering the case $g(s) = s$ and $ k > 0$, one obtains a corresponding equation of elliptic type like (\ref{eq12}). For more detailed  mathematical and physical interpretation of equations like (\ref{eq12}), we refer to \cite{bmmlb,bl,cs,psw,0s} and the references therein.

Problem  (\ref{eq12})  has been studied extensively recently. In \cite{lw,lw2}, with the help of a constrained minimization argument, the existence of positive ground states solution  was proved. Furthermore, Liu and Wang  \cite{lw2}  proved that Eq. (\ref{eq12}) has a ground states solution by using a change of variables and treating the new problem in an Orlicz space when $4\leq p<\frac{4N}{N-2}$ and the potential $V(x)\in C(\mathbb R^{N}, \mathbb R)$ satisfies
$$(V)\inf\limits_{x\in\mathbb R^{N}}V(x)\geq a>0,\ \forall\ \tilde{M}>0,\hbox{meas}\{x\in \mathbb R^{N}\ | \ V(x)\leq\tilde{ M}\}<+\infty.$$
Such kind of hypotheses was firstly introduced by Bartsch and Wang \cite{bw} to guarantee the compactness of embedding  of
$E:=\left\{u\in H^{1}(\mathbb R^{N})\ | \ \int_{\mathbb R^{N}}V(x)u^{2}<\infty\right\}\hookrightarrow L^{s}(\mathbb R^{N}),$
when $ 2\leq s<2^*.$ In \cite{lww},  by utilizing the Nehari method, Liu, Wang and Wang established the existence of  ground state solutions to (\ref{eq12}). In addition, their argument does not depend on any change of variables. But again $4\leq p<\frac{4N}{N-2}$  is assumed. Very recently, in \cite{rs}, Ruiz and Siciliano showed Eq. (\ref{eq12}) has a ground states solution for $N\geq3,$ $2<p<\frac{4N}{N-2}$ via Nehari-Poho\v{z}aev type constraint and concentration-compactness Lemma, moreover, the potential $V (x)$ satisfies the following conditions:
$$(V_{1})\ V\in C^{1}(\mathbb R^{N},\mathbb R^{+}), 0<V_{0}\leq V(x)\leq V_{\infty}=\lim\limits_{|x|\rightarrow\infty}V(x)<+\infty;\ \ \ \  \ $$
$$(V_{2})\ \nabla V(x)\cdot x\in L^{\infty}(\mathbb R^{N});\ \ \ \ \ \  \ \ \ \ \ \ \ \ \  \ \ \ \ \ \ \ \ \ \  \ \ \ \  \ \ \ \  \ \ \ \  \ \  \ \ \ \  \ \ \ \ \ \ \ \ \ \ \ \ \ $$
$$(V_{3})\ \hbox{the map}\ t\mapsto t^{\frac{N+2}{N+p}}V(t^{\frac{1}{N+p}}x)\  \hbox{ is concave for any}\ x\in \mathbb R^{N}.\ \ \ \ \  \ \ \   \ \ \ \ $$
Wu and Wu \cite{ww} obtained the existence of radial solutions for (\ref{eq12}) by using change of variables when $2<p<\frac{4N}{N-2}$ and the potential $V (x)$ satisfies   the similar assumptions as those in Ruiz and Siciliano \cite{rs}.  Based on mountain pass theorem and the Truding-Moser inequality, Moameni \cite{ma} obtained the  existence of nontrivial solutions for (\ref{eq12}) without limit for $|x|\rightarrow\infty$ on the $V(x)$.

For the study of elliptic system, when $k = 0$, there are also several papers concerned with the existence of ground state solutions. In \cite{ll}, under the assumptions:
$$(a_{1})\ 0<\bar{V}_{j}(x),\mu_{j}(x)\ \hbox{are}\ \bar{\tau}_{i}-\hbox{periodic in}\ x_{i}, \bar{\tau}_{i}> 0, 1\leq i\leq N, j=1, 2;$$
$$(a_{2})\ 0<\bar{\beta}(x) \ \hbox{is}\ \bar{\tau}_{i}-\hbox{periodic in}\ x_{i}, \bar{\tau}_{i}> 0, 1\leq i\leq N,\ \ \ \ \ \ \ \ \  \ \ \ \ \  \ \ \ \ \  \ \ $$
 Liu and Liu studied  the following  system
$$
\left\{
     \begin{array}{ll}
\aligned &-\Delta u_{1}+V_{1}(x)u_{1}=\hat{\mu}_{1}(x)u_{1}^{3}+\tilde{\beta}(x)u_{1}u_{2}^{2},\\
&-\Delta u_{2}+V_{2}(x)u_{2}=\hat{\mu}_{2}(x)u_{2}^{3}+\tilde{\beta}(x)u_{1}^{2}u_{2},\\
&\ u_{1}, \ u_2\in H^{1}(\mathbb R^{N}). \endaligned
     \end{array}
   \right.$$
They obtained the existence of ground states solution with the help of variational method. Based on the method of Nehari manifold and the concentration-compactness principle, Zhang, Xu and Zhang \cite{zxz} proved the existence of positive ground states solutions for the nonlinear Schr\"{o}dinger system
$$
\left\{
     \begin{array}{ll}
\aligned &-\Delta u+( \tilde{a}(x)+1)u =F_{u}(u,v)+\tilde{\lambda} u,\\
&-\Delta v+( \tilde{b}(x)+1 ) v =F_{v}(u,v)+\tilde{\lambda} v,\endaligned
     \end{array}
   \right.$$
where $\tilde{a}(x),\tilde{b}(x)$ satisfies
$$(a_{3})\ \tilde{a},\tilde{b}\in L^{\infty}(\mathbb R^{N}),\ \ \inf_{\mathbb R^{N}}\{\tilde{a}(x)+1\}>\tilde{\lambda},\ \inf_{\mathbb R^{N}}\{\tilde{b}(x)+1\}>\tilde{\lambda}; \ \  \ \ \ \ \ \ \  \ \ \ \ \ $$
$$(a_{4})\ \tilde{a}(x)=\tilde{a}(x+y),\ \tilde{ b}(x)=\tilde{b}(x+y),\ \forall\ x\in\mathbb R^{N},\ y\in \mathbb{Z}^N.\ \ \ \ \ \ \  \ \ \  \ \ \ \  \  $$
When $k\neq 0$, Guo and Tang \cite{gt} studied the following quasilinear Schr\"{o}dinger system
$$
\left\{
     \begin{array}{ll}
\aligned &-\Delta u+(\hat{\lambda} \hat{a}(x)+1)u-\frac{1}{2}\triangle(u^{2})u=\frac{2\alpha}{\alpha+\beta}|u|^{\alpha-2}u|v|^{\beta},\\
&-\Delta v+(\hat{\mu }\hat{b}(x)+1)v-\frac{1}{2}\triangle(v^{2})v=\frac{2\beta}{\alpha+\beta}|u|^{\alpha}|v|^{\beta-2}v,\endaligned
     \end{array}
   \right.$$
and proved the existence of positive ground state solution by using concentration-compactness Lemma.  Their argument depends on change of variables and again $4\leq\alpha+\beta<\frac{4N}{N-2}$ is assumed. Moreover, the potential $\hat{a}(x),\hat{b}(x)$ satisfies
$$(a_{5})\ 0\leq \hat{a}(x), \hat{b}(x)\in C(\mathbb R^{N},\mathbb R),\ \Omega:=\hbox{int}\{\hat{a}^{-1}(0)\}=\hbox{int}\{\hat{b}^{-1}(0)\}\ \ \hbox{is} \ \ \ \ \ $$
$$ \hbox{nonempty with smooth boundary and}\ \bar{ \Omega}= \hat{a}^{-1}(0) = \hat{b}^{-1}(0);$$
$$(a_{6})\ \hbox{there exist}\ M_{1},M_{2}>0 \ \ \hbox{such that}\ \ \ \ \ \ \ \  \ \ \ \ \ \ \  \ \ \ \  \ \ \ \ \ \ \ \ \ \ \ \ \ \ \ \ \ \  \ \ \ \ \ \ $$
$$\ \ \ \ meas(\{x\in \mathbb R^{N}\ | \ \hat{a}(x)\leq M_{1}\})<\infty,$$
$$\ \ \ \ meas(\{x\in \mathbb R^{N}\ | \ \hat{b}(x)\leq M_{2}\})<\infty,$$
where $meas$ denotes the Lebesgue measure in $\mathbb R^{N}$. We emphasize that for the single equation  considered in \cite{lw2,lww} the results hold for $|u|^{p-2}u, p\in[4,\frac{4N}{N-2})$;   and for the system in \cite{gt},  the ground states solution is obtained only for $ \alpha+\beta\in[4,\frac{4N}{N-2})$. See also \cite{acm,m,s,su} for related results.  An interesting question is that whether or not we can obtain existence results for (\ref{eq11}) without limit for $|x|\rightarrow\infty$ on the $A(x)$ and $B(x)$? Further, when $A(x)$ and $B(x)$ are periodic potentials, whether  we can find the existence of ground states solutions for (\ref{eq11})?

To the best of our knowledge, there is no results on the existence of positive ground states solutions for periodic  quasilinear Schr\"{o}dinger system  (\ref{eq11}) with $2<\alpha+\beta<\frac{4N}{N-2}$  and thus to prove the existence of positive ground states solutions to problem (\ref{eq11}) is the first purpose of the present paper. Since the strategy in \cite{gt} is only for $\alpha+\beta\in[4,\frac{4N}{N-2})$, we want to use a unified approach to study   (\ref{eq11}) with $\alpha+\beta\in(2,\frac{4N}{N-2})$ and this is the second purpose of the present paper.
\vskip4pt
Before state our main results, we  assume $A,B\in C^{1}(\mathbb R^{N},\mathbb R^{+})$  and satisfy
\vskip4pt
\noindent$(A_{1})$  $A(x)$ is $\tau_{i}$-periodic in $x_{i}$, $\tau_{i}>0$, $1\leq i\leq N,$ $A(x)\geq A_{0}>0;$\\
$(A_{2})$ $  \nabla A(x)\cdot x\in L^{\infty}(\mathbb R^{N}),(\alpha+\beta-2)A(x)-\nabla A(x)\cdot x\geq0$;\\
$(A_{3})$ the map $s\mapsto s^{\frac{N+2}{N+\alpha+\beta}}A(s^{\frac{1}{N+\alpha+\beta}}x)$ is concave for any $x\in \mathbb R^{N}$;\\
$(B_{1})$ $B(x)$ is $\tau_{i}$-periodic in $x_{i}$, $\tau_{i}>0$, $1\leq i\leq N,$   $B(x)\geq B_{0}>0$;\\
$(B_{2})$ $  \nabla B(x)\cdot x\in L^{\infty}(\mathbb R^{N}),(\alpha+\beta-2)B(x)-\nabla B(x)\cdot x\geq0$;\\
$(B_{3})$ the map $s\mapsto s^{\frac{N+2}{N+\alpha+\beta}}B(s^{\frac{1}{N+\alpha+\beta}}x)$ is concave for any $x\in \mathbb R^{N}$.
\vskip4pt

It is worth noting that the assumptions $(A_{1})$ and $(B_{1})$ have been used by many authors in dealing with a single semilinear Schr\"{o}dinger equation and in \cite{ll} to study semilinear elliptic system. The similar hypotheses on $A(x)$ and $B(x)$ as above $(A_{2})$, $(A_{3})$, $(B_{2})$, $(B_{3})$  once appeared in \cite{rs,ww} to study  the single quasilinear Schr\"{o}dinger equation.
 Our main result  reads as follows.

\begin{Theorem}\label{th11}
 Under the assumptions  $(A_{1})$-$(A_{3})$, $(B_{1})$-$(B_{3})$. If$\alpha$, $\beta>1,\alpha+\beta\in(2,\frac{4N}{N-2})$, then problem (\ref{eq11}) has a positive ground states solution.
\end{Theorem}

\begin{Corollary}\label{co12}
If $B(x)$ is a positive constant, one can still obtain the same results as  Theorem \ref{th11} for system (\ref{eq11}).
\end{Corollary}

\noindent {\bf Remark 1.3.}  Our result can be looked on as a complement to results in \cite{gt,lw2,lww}. Moreover, compared with the  hypotheses  in  \cite{rs,ww}, we do not assume that $B(x)$ has a limit for $|x|\rightarrow\infty$. Theorem \ref{th11} also extends the main result in \cite{gt,lw2,lww,ma,rs,ww,zxz} to the quasilinear Schr\"{o}dinger system.
\vskip4pt

The rest of the paper is organized as follows. In Section 2, we state the variational framework of our problem and establish some preliminary results. Theorem \ref{th11} is proved in Section 3.

\mysection {Preliminaries}
Set $X :=\ H\times H$
with  $H := \{ u \in H^{1}(\mathbb R^{N})\   | \ \int_{\mathbb R^{N}}u^{2}|\nabla u|^{2}<+\infty  \},$
where $H^{1}(\mathbb R^{N})$ is the usual Sobolev space. It is well known that $H$ is a complete metric space which is endowed with the distance
$$ d_{H}( u, \omega  ):=\|u-\omega\|_{H^{1}}+|\nabla u^{2}-\nabla \omega^{2}|_{L^{2}}.  $$
We define $$\aligned d_{X}((u,v),(\omega,\nu))
&:=\|u-\omega\|_{H^{1}}+|\nabla u^{2}-\nabla \omega^{2}|_{L^{2}}+\|v-\nu\|_{H^{1}}\\
&\ \ \ \ +|\nabla v^{2}-\nabla \nu^{2}|_{L^{2}}. \endaligned$$
A function $(u,v)\in X$ is called a weak solution of (\ref{eq11}), if for all $\varphi_{1},\varphi_{2}\in C_{0}^{\infty}(\mathbb R^{N})$, it holds
$$\int_{\mathbb R^{N}}\left((1+u^{2})\nabla u\nabla \varphi_{1} +(u|\nabla u|^{2}+A(x)u   -\frac{2\alpha}{\alpha+\beta}|u|^{\alpha-2}u|v|^{\beta})\varphi_{1}\right)=0$$
and
$$\int_{\mathbb R^{N}}\left((1+v^{2})\nabla v\nabla \varphi_{2}+ (v|\nabla v|^{2}+B(x)v
 -\frac{2\beta}{\alpha+\beta}|u|^{\alpha} |v|^{\beta-2}v )\varphi_{2}\right)=0.$$
Hence there is a one-to-one correspondence between solutions of (\ref{eq11}) and critical points of the following  functional $I: X \rightarrow\mathbb R$ defined by
\begin{equation}\label{eq21}
\aligned I(u,v) &=\int_{\mathbb R^{N}}\bigg(\frac{1}{2}(|\nabla u|^{2}+|\nabla v|^{2}+A(x)u^{2}+B(x)v^{2}+u^{2}|\nabla u|^{2}+v^{2}|\nabla v|^{2})\\
&\ \ \ \ -\frac{2}{\alpha+\beta} |u|^\alpha |v|^\beta \bigg). \endaligned
\end{equation}
It is easy to check that the functional $I$ is continuous on $X.$ In addition, for any $(\varphi_{1},\varphi_{2})\in C_{0}^{\infty}(\mathbb R^{N})$, $(u,v)\in X,$ and $(u,v)+(\varphi_{1},\varphi_{2})\in X$, we can compute the Gateaux derivative
$$\aligned &\ \langle I'(u,v),(\varphi_{1},\varphi_{2})\rangle
= \int_{\mathbb R^{N}}\bigg( (1+u^{2})\nabla u\nabla \varphi_1 +(1+v^{2})\nabla v\nabla \varphi_{2}+u|\nabla u|^2\varphi_{1}\\
&\ +v|\nabla v|^{2}\varphi_{2}+A(x)u\varphi_{1}+B(x)v\varphi_2 -\frac{2\alpha|u|^{\alpha-2}u|v|^{\beta}\varphi_{1}}{\alpha+\beta}
  -\frac{2\beta|v|^{\beta-2}v|u|^{\alpha}\varphi_{2}}{\alpha+\beta}\bigg).\endaligned$$
Then, $(u,v)\in X $ is a solution of (\ref{eq11}) if and only if
$\langle I'(u,v),(\varphi_{1},\varphi_{2})\rangle=0.$

For convenience, throughout this paper, $C$ and $ C_{i}(i=1,2,\ldots)$ denote (possibly different) positive constants, $\int_{\mathbb R^{N}}g$ denotes the integral $\int_{\mathbb R^{N}}g(z)dz$ and $\|(u,v)\|_{H}:=\left(\int_{\mathbb R^{N}}(|\nabla u|^{2}+|\nabla v|^{2}+u^{2}+v^{2})\right)^{1/2}$.
\vskip4pt
\begin{Lemma}\label{le21} Suppose $(A_{1})$-$(A_{3})$ and $(B_{1})$-$(B_{3})$ are satisfied, if $(u,v)\in X$ is a weak solution of problem (\ref{eq11}), then $(u,v)$ satisfies $P(u,v)=0,$ where
$$ \aligned P(u,v):&=\frac{N-2}{2}\int_{\mathbb R^{N}}|\nabla u|^{2}+\frac{N-2}{2}\int_{\mathbb R^{N}}|\nabla v|^{2}+\frac{N}{2}\int_{\mathbb R^{N}}A(x)u^{2}\\
&\ \ \ \  +\frac{N}{2}\int_{\mathbb R^{N}}B(x)v^{2}+\frac{1}{2}\int_{\mathbb R^{N}}\nabla A(x)\cdot xu^{2}+\frac{1}{2}\int_{\mathbb R^{N}}\nabla B(x)\cdot xv^{2}\\
&\ \ \ \ +\frac{N-2}{2}\int_{\mathbb R^{N}}u^{2}|\nabla u|^{2}+\frac{N-2}{2}\int_{\mathbb R^{N}}v^{2}|\nabla v|^{2}-\frac{2N}{\alpha+\beta}\int_{\mathbb R^{N}}|u|^{\alpha}|v|^{\beta}.\endaligned $$
\end{Lemma}

\begin{proof}   The proof is standard, so we omit it here.
\end{proof}
\vskip4pt
Next, for any $u\in H$, we define $u_{t}:\mathbb R^{+}\rightarrow H$ by
$$ u_{t}(x):=tu(t^{-1}x).$$
Let $t\in \mathbb R^{+}$ and $(u,v)\in X$, we have that
$$\aligned I(u_{t},v_{t})&=\frac{t^{N}}{2}\int_{\mathbb R^{N}}|\nabla u|^{2}+\frac{t^{N}}{2}\int_{\mathbb R^{N}}|\nabla v|^{2}+\frac{t^{N+2}}{2}\int_{\mathbb R^{N}}A(tx)u^{2}\\
&\ \ \ \ +\frac{t^{N+2}}{2}\int_{\mathbb R^{N}}B(tx)v^{2}+\frac{t^{N+2}}{2}\int_{\mathbb R^{N}}u^{2}|\nabla u|^{2}\\
&\ \ \ \ +\frac{t^{N+2}}{2}\int_{\mathbb R^{N}}v^{2}|\nabla v|^{2}-\frac{2t^{N+\alpha+\beta}}{\alpha+\beta}\int_{\mathbb R^{N}}|u|^{\alpha}|v|^{\beta}.\endaligned$$
Denote $h_{uv}(t):=I(u_{t},v_{t})$. Since $\alpha+\beta>2,$ we see that $h_{uv}(t)>0$ for $t>0$ small enough and $h_{uv}(t)\rightarrow -\infty$ as $t\rightarrow +\infty,$ this implies that $h_{uv}(t)$ attains its maximum. Moreover, thanks to $(A_{2}), (B_{2})$, $h_{uv}(t): \mathbb R^{+}\rightarrow \mathbb R$ is $C^{1}$ and
$$\aligned h'_{uv}(t)
&=\frac{N}{2}t^{N-1}\int_{\mathbb R^{N}}(|\nabla u|^{2}+|\nabla v|^{2})+\frac{N+2}{2}t^{N+1}\int_{\mathbb R^{N}} (A(tx)u^{2}\\
&\ \ \ \ +B(tx)v^{2}+u^{2}|\nabla u|^{2}+v^{2}|\nabla v|^{2} )\\
&\ \ \ \ +\frac{ t^{N+1}}{2}\int_{\mathbb R^{N}} (\nabla A(tx)\cdot txu^{2}+\nabla B(tx)\cdot txv^{2} )\\
&\ \ \ \ -\frac{2(N+\alpha+\beta)}{\alpha+\beta}t^{N+\alpha+\beta-1}\int_{\mathbb R^{N}}|u|^{\alpha}|v|^{\beta}.\endaligned $$
\vskip4pt
Motivated by \cite{rs}, we introduce the following set
$$ \mathcal{M}:=\{(u,v)\in X\backslash\{(0,0)\}\ | \ G(u,v)=0\},$$
where $G:X\rightarrow \mathbb R$ is defined as
$$ \aligned G(u,v)&:=\frac{N}{2} \int_{\mathbb R^{N}}(|\nabla u|^{2}+|\nabla v|^{2})+\frac{N+2}{2} \int_{\mathbb R^{N}}(A(x)u^{2}+B(x)v^{2})\\
&\ \ \ \ +\frac{N+2}{2} \int_{\mathbb R^{N}}(u^{2}|\nabla u|^{2}+v^{2}|\nabla v|^{2})+\frac{1}{2}\int_{\mathbb R^{N}}\nabla A(x)\cdot xu^{2}\\
&\ \ \ \ +\frac{1}{2}\int_{\mathbb R^{N}}\nabla B(x)\cdot xv^{2}-\frac{2(N+\alpha+\beta)}{\alpha+\beta}\int_{\mathbb R^{N}}|u|^{\alpha}|v|^{\beta}. \endaligned $$
Note that if $(u,v)\in X$ is a weak solution of problem (\ref{eq11}), then $\langle I'(u,v),(u,v)\rangle=0$ and $P(u,v)=0$, therefore, $G(u,v)=0.$

Define
\begin{equation}\label{eq22}
m:=\inf\limits_{(u,v)\in \mathcal{M}}I(u,v).
\end{equation}
Then our aim is to prove that $m$ is achieved. In the rest of this section, we will give some properties of the set $\mathcal{M}$ and show that (\ref{eq22}) is well defined.
\vskip4pt
\begin{Lemma}\label{le22} Suppose that 
     $(u,v)\in X\backslash\{(0,0)\}$, then there is a unique 
     $\bar{t}:=t(u,v)>0$ such that $h_{uv}$ attains its maximum 
     at $\bar{t}$ and
      $m=\inf\limits_{(u,v)\in X}\max\limits_{t>0}I(u_{t},v_{t}).$ Moreover,
       if $G(u,v)<0$, then $\bar{t}\in(0,1)$.
\end{Lemma}
\begin{proof}
As mentioned early, $h_{uv}(t)$ attains its maximum. Let us make the change of variable  $s=t^{N+\alpha+\beta}$, thus
$$\aligned h_{uv}(s)&=\frac{s^{\frac{N}{N+\alpha+\beta}}}{2}\int_{\mathbb R^{N}}|\nabla u|^{2}+\frac{s^{\frac{N}{N+\alpha+\beta}}}{2}\int_{\mathbb R^{N}}|\nabla v|^{2}+\frac{s^{\frac{N+2}{N+\alpha+\beta}}}{2}\int_{\mathbb R^{N}}u^{2}|\nabla u|^{2}\\
&\ \ \ \ +\frac{s^{\frac{N+2}{N+\alpha+\beta}}}{2}\int_{\mathbb R^{N}}v^{2}|\nabla v|^{2}+\frac{s^{\frac{N+2}{N+\alpha+\beta}}}{2}\int_{\mathbb R^{N}}A(s^{\frac{1}{N+\alpha+\beta}}x)u^{2}\\
&\ \ \ \ +\frac{s^{\frac{N+2}{N+\alpha+\beta}}}{2}\int_{\mathbb R^{N}}B(s^{\frac{1}{N+\alpha+\beta}}x)v^{2}-\frac{2s }{\alpha+\beta}\int_{\mathbb R^{N}}|u|^{\alpha}|v|^{\beta}.\endaligned $$
By  $(A_{3})$ and $(B_{3})$, $h_{uv}(s)$ is a concave function. We already know that it attains its maximum, let $\bar{t}$ be the unique point at which this maximum is achieved. We claim that $\bar{t}$ is the unique critical point of $h_{uv}$. Indeed, if $h_{uv}$ has another maximum point which is different from $\bar{t}$. Let $h_{uv}(\bar{s})\geq h_{uv}(\bar{t})$, because $h_{uv}$ is a strictly concave function, it has for any $\lambda\in (0,1)$, we have
$ h_{uv}(\lambda\bar{t}+(1-\lambda)\bar{s})>\lambda h_{uv}(\bar{t})+(1-\lambda) h_{uv}(\bar{s})\geq h_{uv}(\bar{t}),$
so for any $\delta>0$, as long as $\lambda$ is sufficiently close to 1, there is always $t=\lambda\bar{t}+(1-\lambda)\bar{s}$. But $h_{uv}(t)>h_{uv}(\bar{t})$, which is in contradiction with $\bar{t}$ being the maximum point of $h_{uv}$. Hence, $\bar{t}$ is the unique critical point of $h_{uv}$. Then $\bar{t}$ is the unique critical point of $h_{uv}$ and $h_{uv}$ is positive and increasing for $0<t <\bar{t}$ and decreasing for $t>\bar{t}$. In particular, for any $u,v\neq0$, $\bar{t}\in \mathbb R$ is the unique value such that $(u_{\bar{t}},u_{\bar{t}})$ belongs to $\mathcal{M},$ and $I(u_{\bar{t}},v_{\bar{t}})$ reaches a global maximum for $t=\bar{t}$. Moreover, we claim that $\bar{t}\in(0,1)$. Indeed, if $G(u,v)<0$, it follows from
$$ \aligned G(u,v)&:=\frac{N}{2} \int_{\mathbb R^{N}}(|\nabla u|^{2}+|\nabla v|^{2})+\frac{N+2}{2} \int_{\mathbb R^{N}}(A(x)u^{2}+B(x)v^{2}+u^{2}|\nabla u|^{2}\\
&\ \ \ \ +v^{2}|\nabla v|^{2})+\frac{1}{2}\int_{\mathbb R^{N}}(\nabla A(x)\cdot xu^{2}+\nabla B(x)\cdot xv^{2})\\
&\ \ \ \ -\frac{2(N+\alpha+\beta)}{\alpha+\beta}\int_{\mathbb R^{N}}|u|^{\alpha}|v|^{\beta}<0,\endaligned $$
and
$$\aligned &\frac{N}{2}\bar{t}^{N}\int_{\mathbb R^{N}}(|\nabla u|^{2}+|\nabla v|^{2})+\frac{N+2}{2}\bar{t}_{1}^{N+2}\int_{\mathbb R^{N}}(A(\bar{t}x)u^{2}+B(\bar{t}x)v^{2}+u^{2}|\nabla u|^{2}\\
&  +v^{2}|\nabla v|^{2})+\frac{\bar{t}^{N+2}}{2}\int_{\mathbb R^{N}}((\nabla A(\bar{t}x)\cdot \bar{t}x)u^{2}+(\nabla B(\bar{t}x)\cdot \bar{t}_{1}x)v^{2})\\
& -\frac{2(N+\alpha+\beta)}{\alpha+\beta}\bar{t}^{N+\alpha+\beta }\int_{\mathbb R^{N}}|u|^{\alpha}|v|^{\beta}=0,\endaligned $$
that
$$\aligned &\frac{N}{2}\left(\bar{t}^{N+\alpha+\beta }-\bar{t}^{N}\right)\int_{\mathbb R^{N}}(|\nabla u|^{2}+|\nabla v|^{2})+\frac{N+2}{2}\left(\bar{t}^{N+\alpha+\beta }-\bar{t}_{1}^{N+2}\right)\int_{\mathbb R^{N}}(A(\bar{t}x)u^{2}\\
&+B(\bar{t}x)v^{2}+u^{2}|\nabla u|^{2}+v^{2}|\nabla v|^{2})+\frac{\bar{t}^{N+\alpha+\beta}-\bar{t}^{N+2}}{2} \int_{\mathbb R^{N}}((\nabla A(\bar{t}x)\cdot \bar{t}x)u^{2}\\
&+(\nabla B(\bar{t}x)\cdot \bar{t} x)v^{2})=0,\endaligned $$
which implies that $\bar{t}<1$. This finishes the proof.
\end{proof}
It follows from Lemma \ref{le22}  that $\mathcal{M}\neq\emptyset $ and (\ref{eq22}) is well defined.
\vskip4pt
\begin{Lemma}\label{le23}
 $m>0$.
\end{Lemma}

\begin{proof}  Let us define
$$\aligned \bar{I}(u,v)&=\int_{\mathbb R^{N}}\bigg(\frac{1}{2}(|\nabla u|^{2}+|\nabla v|^{2}+ A_{0}u^{2}+B_{0}v^{2}\\
&\ \ \ \ +u^{2}|\nabla u|^{2}+v^{2}|\nabla v|^{2}) -\frac{2}{\alpha+\beta} |u|^{\alpha}|v|^{\beta}\bigg),\endaligned$$
where $A_{0},B_{0}$ come from $(A_{1}),(B_{1})$. Obviously, $\bar{I}(u,v)\leq I(u,v)$, which implies that for any $u,v\neq0$,
$$\bar{m}:=\inf_{(u,v)\in X}\max_{t>0}\bar{I}(u_{t},v_{t})\leq \inf_{(u,v)\in X}\max_{t>0}I(u_{t},v_{t})=m.$$
Then, it suffices to prove that $\bar{m}>0.$ Define
$$ \bar{\mathcal{M}}=\{(u,v)\in X\backslash\{(0,0)\} \ | \ g'_{uv}(1)=0\},$$
where $g_{uv}(t)=\bar{I}(u_{t},v_{t}).$ By Lemma \ref{le22} applied to $A\equiv A_{0},B\equiv B_{0}$, we get that
$$\bar{m }=\inf_{(u,v)\in\bar{\mathcal{M}}}\bar{I}(u,v).$$
For every $(u,v)\in\bar{\mathcal{ M}}$, by using H\"{o}lder, Young and interpolation inequalities,
$$\aligned  &\ \frac{N+2}{2}\int_{\mathbb R^{N}}(A_{0}u^{2}+B_{0}v^{2})+\frac{N+2}{2}\int_{\mathbb R^{N}}(u^{2}|\nabla u|^{2}+v^{2}|\nabla v|^{2})\\
\leq&\ \frac{2(N+\alpha+\beta)}{\alpha+\beta}\int_{\mathbb R^{N}}|u|^{\alpha}|v|^{\beta}\\
\leq&\ \frac{2(N+\alpha+\beta)}{\alpha+\beta}\bigg(\int_{\mathbb R^{N}}|u |^{\alpha+\beta}\bigg)^{\frac{\alpha}{\alpha+\beta}}\bigg(\int_{\mathbb R^{N}}|v |^{\alpha+\beta}\bigg)^{\frac{\beta}{\alpha+\beta}}\\
\leq&\ \frac{2\alpha(N+\alpha+\beta)}{(\alpha+\beta)^{2}} \int_{\mathbb R^{N}}|u |^{\alpha+\beta} +\frac{2\beta(N+\alpha+\beta)}{(\alpha+\beta)^{2}} \int_{\mathbb R^{N}}|v |^{\alpha+\beta} \\
\leq&\ \frac{2\alpha(N+\alpha+\beta)}{(\alpha+\beta)^{2}}\bigg(\int_{\mathbb R^{N}}u^{2}\bigg)^{\frac{l(\alpha+\beta)}{2}}\bigg(\int_{\mathbb R^{N}}(|u|^{\frac{4N}{N-2}}\bigg)^{\frac{(1-l)(N-2)(\alpha+\beta)}{4N}}\\
&\  +\frac{2\beta(N+\alpha+\beta)}{(\alpha+\beta)^{2}}\bigg(\int_{\mathbb R^{N}}v^{2}\bigg)^{\frac{l(\alpha+\beta)}{2}}\bigg(\int_{\mathbb R^{N}}(|v|^{\frac{4N}{N-2}}\bigg)^{\frac{(1-l)(N-2)(\alpha+\beta)}{4N}}\\
\leq&\ \frac{N+2}{2}\int_{\mathbb R^{N}}(A_{0}u^{2}+B_{0}v^{2})+C\int_{\mathbb R^{N}}(|u|^{\frac{4N}{N-2}}+|v|^{\frac{4N}{N-2}}).\endaligned$$
 So, by the Sobolev inequality,
$$\aligned \frac{N+2}{2}\int_{\mathbb R^{N}}(u^{2}|\nabla u|^{2}+v^{2}|\nabla v|^{2})&\leq C\int_{\mathbb R^{N}}(|u|^{\frac{4N}{N-2}}+|v|^{\frac{4N}{N-2}})\\
&\leq C_{1}\bigg(\int_{\mathbb R^{N}}u^{2}|\nabla u|^{2}+v^{2}|\nabla v|^{2}\bigg)^{\frac{N}{N-2}},\endaligned $$
which implies that $\int_{\mathbb R^{N}}(u^{2}|\nabla u|^{2}+v^{2}|\nabla v|^{2})$ is bounded away from zero on $\bar{\mathcal{ M}}$, then
$$\aligned\bar{I}(u,v)&= \frac{\alpha+\beta}{2(N+\alpha+\beta)}\int_{\mathbb R^{N}}|\nabla u|^{2}+\frac{\alpha+\beta}{2(N+\alpha+\beta)}\int_{\mathbb R^{N}}|\nabla v|^{2}\\
&\ \ \ \ +\frac{A_{0}(\alpha+\beta-2)}{2(N+\alpha+\beta)}\int_{\mathbb R^{N}}u^{2}+\frac{B_{0}(\alpha+\beta-2)}{2(N+\alpha+\beta)}\int_{\mathbb R^{N}}v^{2}\\
&\ \ \ \ +\frac{\alpha+\beta-2}{2(N+\alpha+\beta)}\int_{\mathbb R^{N}}u^{2}|\nabla u|^{2}+\frac{\alpha+\beta-2}{2(N+\alpha+\beta)}\int_{\mathbb R^{N}}v^{2}|\nabla v|^{2}\\
&>0.\endaligned $$
\end{proof}
\vskip4pt
The following result implies that $I$ is coercive on $\mathcal{M}$, which ensures that any minimizing sequence for $m$ is bounded.
\begin{Lemma}\label{le24}
 There exists $c> 0$ such that for any $(u,v)\in \mathcal{M},$
$$  I(u,v)\geq c\int_{\mathbb R^{N}}(|\nabla u|^{2}+|\nabla v|^{2}+u^{2}+ v^{2}+ u^{2}|\nabla u|^{2}+v^{2}|\nabla v|^{2}).$$
\end{Lemma}

\begin{proof} Take $(u,v)\in\mathcal{ M}$ and choose $t\in(0, 1).$  By a direct computation, there holds
$$\aligned &\ I(u_{t},v_{t})-t^{N+\alpha+\beta}I(u,v)\\
=& \ \int_{\mathbb R^{N}}\bigg(\frac{t^{N}}{2}-\frac{t^{N+\alpha+\beta}}{2}\bigg)(|\nabla v|^{2}+|\nabla v|^{2})\\
&\ +\int_{\mathbb R^{N}}\bigg(\frac{t^{N+2}}{2}-\frac{t^{N+\alpha+\beta}}{2}\bigg)(u^{2}|\nabla u|^{2}+v^{2}|\nabla v|^{2})\\
&\ +\int_{\mathbb R^{N}}\bigg(\frac{t^{N+2}}{2}A(tx)-\frac{t^{N+\alpha+\beta}}{2}A(x)\bigg)u^{2}\\
&\ +\int_{\mathbb R^{N}}\bigg(\frac{t^{N+2}}{2}B(tx)-\frac{t^{N+\alpha+\beta}}{2}B(x)\bigg)v^{2}.\endaligned$$
By $(A_{1}),(B_{1})$, there exist $A_{1},B_{1}$ such that $A(tx)\geq A_{0}\geq\delta_{1} A_{1}\geq \delta_{1} A(x)$ for  $\delta_{1}\in(0, 1)$ depending only on $A_{0}$ and $A_{1}$, $B(tx)\geq B_{0}\geq\delta_{2}B_{1}\geq \delta_{2} B(x)$ for  $\delta_{2}\in(0, 1)$ depending only on $B_{0}$ and $B_{1}$. By choosing a smaller $t$,  we get that
$$\frac{t^{N+2}}{2}A(tx)-\frac{t^{N+\alpha+\beta}}{2}A(x)\geq \bigg(\delta_{1}\frac{t^{N+2}}{2}-\frac{t^{N+\alpha+\beta}}{2}\bigg)A(x)\geq\gamma_{1},  $$
$$\frac{t^{N+2}}{2}B(tx)-\frac{t^{N+\alpha+\beta}}{2}B(x)\geq \bigg(\delta_{2}\frac{t^{N+2}}{2}-\frac{t^{N+\alpha+\beta}}{2}\bigg)B(x)\geq\gamma_{2} $$
for  positive fixed constant $\gamma=\min\{\gamma_{1},\gamma_{2}\}> 0$. By Lemma \ref{le22}, $I(u_{t},v_{t})\leq I(u,v)$ and then
$$\aligned &\ (1-t^{N+\alpha+\beta})I(u,v)\\
\geq&\ I(u_{t},v_{t})-t^{N+\alpha+\beta}I(u,v)\\
\geq& \ \gamma\int_{\mathbb R^{N}}\left(|\nabla u|^{2}+|\nabla v|^{2}+u^{2}+ v^{2}+u^{2}|\nabla u|^{2}+v^{2}|\nabla v|^{2}\right).\endaligned$$
We conclude by taking a smaller $\gamma$ and choosing $c=\frac{\gamma }{1-t^{N+\alpha+\beta}}$.
\end{proof}

\mysection{Proof of Theorem \ref{th11} }

 In this section, we prove the existence of positive ground states solutions  to (\ref{eq11}). We recall here the result due to Lions (Lemma 1.1 of \cite{l}, part 2) and establish the following lemma.

\vskip4pt
\begin{Lemma}\label{le31}
Let $ r>0, \ q\in[2,\frac{4N}{N-2}).$ If $ \{u_{n}\}$ is bounded in $H$ and
$$ \lim_{n\rightarrow\infty}\sup_{y\in \mathbb R^{N}}\int_{B_{y}(r)}|u_{n}|^{q}=0,$$
then we have $u_{n}\rightarrow0$ in $ L^{p}(\mathbb R^{N})$ for $p\in(2,\frac{4N}{N-2}).$
\end{Lemma}

\begin{proof}  The proof is similar to the proof of  \cite[Lemma 2.2]{wz}.
\end{proof}
\vskip4pt
\begin{Lemma}\label{le32}
 Let  $u_{n}\rightharpoonup u, v_{n}\rightharpoonup v$  in $H^{1}(\mathbb R^{N})$, $u_{n}\rightarrow u, v_{n}\rightarrow v$ a.e in $\mathbb R^{N}$. Then
$$
\lim_{n\rightarrow\infty}\int_{\mathbb R^{N}}|u_{n}|^{\alpha}|v_{n}|^{\beta}-\int_{\mathbb R^{N}}|u|^{\alpha}|v|^{\beta}
=\lim_{n\rightarrow\infty}\int_{\mathbb R^{N}}|u_{n}-u|^{\alpha}|v_{n}-v|^{\beta}.
$$
\end{Lemma}

\begin{proof}
For $n=1,\ 2,\ldots$, we have that
$$\begin{aligned}
&\ \int_{\mathbb R^{N}}|u_n|^\alpha |v_n|^\beta -\int_{\mathbb R^{N}}|u_n-u|^\alpha |v_n-v|^\beta\\
=&\int_{\mathbb R^{N}}(|u_n|^\alpha-|u_n-u|^\alpha) |v_n|^\beta+\int_{\mathbb R^{N}}|u_n-u|^\alpha (|v_n|^\beta-|v_n-v|^\beta).
\end{aligned}$$
Since $u_n\rightharpoonup u, v_n\rightharpoonup v$ in $H^1 (\mathbb{R}^N)$, from \cite[Lemma 2.5]{mj}, one has
$$\int_{\mathbb R^{N}}(|u_n|^\alpha-|u_n-u|^\alpha-|u|^\alpha)^{\frac{p}{\alpha}}\rightarrow0,\ \ n\rightarrow\infty,$$
which means that $|u_n|^\alpha-|u_n-u|^\alpha \rightarrow |u|^\alpha $  in $L^{\frac{p}{\alpha}} (\mathbb{R}^N)$.
Using $|v_n|^\beta \rightharpoonup |v|^\beta$ in $L^{\frac{p}{\beta}} (\mathbb{R}^N)$, it follows from $\alpha +\beta = p$ that
$$\int_{\mathbb R^{N}}(|u_n|^\alpha-|u_n-u|^\alpha) |v_n|^\beta\rightarrow\int_{\mathbb R^{N}}|u|^\alpha |v|^\beta,\ \ n\rightarrow\infty.$$

Similarly, $|v_n|^\beta-|v_n-v|^\beta \rightarrow |v|^\beta $  in $L^{\frac{p}{\beta}} (\mathbb{R}^N)$.
As $|u_n-u|^\alpha \rightharpoonup 0$ in $L^{\frac{p}{\alpha}} (\mathbb{R}^N)$, we obtain that
$$\int_{\mathbb R^{N}}|u_n-u|^\alpha (|v_n|^\beta-|v_n-v|^\beta)\rightarrow0,\ \ n\rightarrow\infty.$$
This proves the lemma.
\end{proof}
\vskip4pt
The following Lemma \ref{le33} is due to  Poppenberg, Schmitt and Wang  from \cite[Lemma 2]{psw}.
\begin{Lemma}\label{le33}
Assume that  $u_{n}\rightharpoonup u$ in $H^{1}(\mathbb R^{N})$. Then
\begin{equation}\label{eq3.1}
\aligned\liminf_{n\rightarrow\infty}\int_{\mathbb R^{N}} u_{n}^{2}|\nabla u_{n}|^{2} &\geq \liminf_{n\rightarrow\infty}\int_{\mathbb R^{N}}(u_{n}-u)^{2}|\nabla u_{n}-\nabla u)|^{2}\\
&\ \ \ \ +\int_{\mathbb R^{N}} u^{2}|\nabla u|^{2}.\endaligned
\end{equation}
\end{Lemma}
\vskip4pt

\begin{Lemma}\label{le34}
$m$ is achieved at some $(u,v)\in\mathcal{ M}$.
\end{Lemma}

\begin{proof} Let $(u_{n},v_{n})\in\mathcal{M}$ so that $I(u_{n},v_{n})\rightarrow m$. By Lemma \ref{le24}, $\{u_{n}\}$, $\{v_{n}\}$, $\{u_{n}^{2}\}$  and $\{v_{n}^{2}\}$ are bounded in $H^{1}(\mathbb R^{N}).$ Then, there exist subsequence of $\{u_{n}\}$, $\{v_{n}\}$ (still denoted by $\{u_{n}\}$, $\{v_{n}\}$) such that $u_{n}\rightharpoonup u$ and $u_{n}^{2}\rightharpoonup u^{2}$ in $H^{1}(\mathbb R^{N})$, $v_{n}\rightharpoonup v$ and $v_{n}^{2}\rightharpoonup v^{2}$ in $H^{1}(\mathbb R^{N})$.
This implies in particular that $\{u_{n}\}$ and $\{v_{n}\}$ are bounded in $L^{\alpha+\beta}(\mathbb R^{N})$, $\alpha+\beta\in(2,\frac{4N}{N-2})$. The proof consists of three steps.
\vskip4pt
{\bf Step 1.} $\int_{\mathbb R^{N}}|u_{n}|^{\alpha}|v_{n}|^{\beta}\not\rightarrow0$.
\vskip4pt
It follows from Lemma \ref{le24}  and
$$ \aligned I(u_{n},v_{n})&=\int_{\mathbb R^{N}}\bigg(\frac{1}{2}(|\nabla u_{n}|^{2}+|\nabla v_{n}|^{2}+A(x)u_{n}^{2}+B(x)v_{n}^{2}\\
&\ \ \ \  +u_{n}^{2}|\nabla u_{n}|^{2}+v_{n}^{2}|\nabla v_{n}|^{2})-\frac{2}{\alpha+\beta} |u_{n}|^{\alpha}|v_{n}|^{\beta}\bigg)\\
&\rightarrow m>0\endaligned$$
that $\|u_{n}\|_{H^{1}}+\|v_{n}\|_{H^{1}}+\|u_{n}^{2}\|_{H^{1}}+\|v_{n}^{2}\|_{H^{1}}\not\rightarrow0.$ By using Lemma \ref{le22}, for any $t > 1$,
$$\aligned m&\leftarrow I(u_{n},v_{n})\\
&\geq I((u_{n})_{t},(v_{n})_{t})\\
&=\frac{t^{N}}{2}\int_{\mathbb R^{N}}(|\nabla u_{n}|^{2}+|\nabla v_{n}|^{2})+\frac{t^{N+2}}{2}\int_{\mathbb R^{N}}(A(tx)u_{n}^{2}+B(tx)v_{n}^{2})\\
&\ \ \ \ +\frac{t^{N+2}}{2}\int_{\mathbb R^{N}}\left(u_{n}^{2}|\nabla u_{n}|^{2}+v_{n}^{2}|\nabla v_{n}|^{2}\right)-\frac{2t^{N+\alpha+\beta}}{\alpha+\beta}\int_{\mathbb R^{N}}|u_{n}|^{\alpha}|v_{n}|^{\beta}\\
&\geq \frac{t^{N}}{2}\int_{\mathbb R^{N}}(|\nabla u_{n}|^{2}+|\nabla v_{n}|^{2}+A_{0}u_{n}^{2} +B_{0}v_{n}^{2}+u_{n}^{2}|\nabla u_{n}|^{2}+v_{n}^{2}|\nabla v_{n}|^{2})\\
&\ \ \ \ -\frac{2t^{N+\alpha+\beta}}{\alpha+\beta}\int_{\mathbb R^{N}}|u_{n}|^{\alpha}|v_{n}|^{\beta}\\
&\geq \frac{t^{N}}{2}\delta-\frac{2t^{N+\alpha+\beta}}{\alpha+\beta}\int_{\mathbb R^{N}}|u_{n}|^{\alpha}|v_{n}|^{\beta},\endaligned$$
where $\delta$ is a fixed constant. It suffices to take $t > 1$ so that $ \frac{t^{N}\delta}{2}>2m$ to get a lower bound for $\int_{\mathbb R^{N}}|u_{n}|^{\alpha}|v_{n}|^{\beta}$.

Therefore, we may assume that
\begin{equation}\label{eq3.2}
\int_{\mathbb R^{N}}|u_{n}|^{\alpha}|v_{n}|^{\beta}\rightarrow D\in(0,\infty).
\end{equation}
\vskip4pt
{\bf Step 2.} Splitting by concentration-compactness.
\vskip4pt
By using (\ref{eq3.2}) and Lemma \ref{le31}, there exist $\delta>0$ and $\{x_{n}\}\subset \mathbb R^{N}$  such that
\begin{equation}\label{eq3.3}
\int_{B_{x_{n}}(1)}|u_{n}|^{\alpha+\beta}>\delta>0.
\end{equation}
Let $\eta_{R}(t)$ a smooth function defined on $[0,+\infty)$ such that
\vskip4pt
a) $\eta_{R}(t)=1$ for $0\leq t\leq R;$

b) $\eta_{R}(t)=0$ for $t\geq2R$;

c) $\eta'_{R}(t)\leq\frac{2}{R}.$
\vskip4pt
Define $ \theta_{n}(x)=\eta_{R}(|x-x_{n}|)u_{n}(x)$, $\lambda_{n}(x)=(1-\eta_{R}(|x-x_{n}|))u_{n}(x)$. Obviously, $ u_{n} = \theta_{n}+\lambda_{n} $. Observe that in particular
\begin{equation}\label{eq3.4}
\liminf_{n\rightarrow+\infty}\int_{B_{x_{n}}(R)}| \theta_{n}|^{\alpha+\beta}\geq\delta.
\end{equation}
\vskip4pt
{\bf Step 3.}  The infimum of  $I|_{\mathcal{M}}$ is achieved.
\vskip4pt
Since  $u_{n}\rightharpoonup u,$ $u_{n}^{2}\rightharpoonup u^{2}$ in $H^{1}(\mathbb R^{N})$. We take $x_{n}\in \tau_{1}\mathbb{Z}\times \cdots \tau_{N}\mathbb{Z}$, where $x_{n}$ is given in (\ref{eq3.3}), define $\omega_{n}=u_{n}(\cdot+x_{n}),\nu_{n}=v_{n}(\cdot+x_{n})$, then $\omega_{n}\rightharpoonup \omega, \omega_{n}^{2}\rightharpoonup \omega^{2}$ in $H^{1}(\mathbb R^{N})$. In this case, by $(A_{1})$ and $(B_{1})$,   $I(u_{n},v_{n})=I(\omega_{n},\nu_{n})$. Using (\ref{eq3.4}), we obtain
$$ \delta<\liminf_{n\rightarrow+\infty}\int_{\mathbb R^{N}}| \theta_{n}|^{\alpha+\beta}\leq\lim_{n\rightarrow+\infty}\int_{B_{0}(2R)}|\omega_{n}|^{\alpha+\beta}=\int_{B_{0}(2R)}|\omega|^{\alpha+\beta},$$
which implies $ \omega \neq 0$, and then $u\neq 0$. We claim that $(u,v)\in \mathcal{M}.$ Indeed, if $(u,v)\not\in \mathcal{M},$ we discuss two cases:
\vskip4pt
{\bf Case 1:} $G(u,v)<0$. By Lemma \ref{le22}, there exists $t\in(0,1)$ such that $(u _{t},v_{t})\in \mathcal{M},$ it follows from $(A_{2})$, $(B_{2})$, $(u_{n},v_{n})\in \mathcal{M}$ and Fatou's Lemma that
$$\aligned m&=\liminf_{n\rightarrow+\infty}\bigg(I( u_{n} ,v_{n} )-\frac{1}{N+\alpha+\beta}G(u_{n} ,v_{n})\bigg)\\
&=\frac{1}{2(N+\alpha+\beta)}\liminf_{n\rightarrow+\infty} \int_{\mathbb R^{N}}\bigg((\alpha+\beta)(|\nabla u_{n}|^{2}+ |\nabla v_{n}|^{2})+((\alpha+\beta-2)A(x)\\
&\ \ \ \ -\nabla A(x)\cdot x)u_{n}^{2}+((\alpha+\beta-2)B(x)-\nabla B(x)\cdot x)v_{n}^{2}\\
&\ \ \ \ +(\alpha+\beta-2)(u_{n}^{2}|\nabla u_{n}|^{2}+v_{n}^{2}|\nabla v_{n}|^{2})\bigg)\\
&\geq \frac{1}{2(N+\alpha+\beta)}\int_{\mathbb R^{N}}\bigg((\alpha+\beta)(|\nabla u|^{2}+ |\nabla v|^{2})+((\alpha+\beta-2)A(x)\\
&\ \ \ \ -\nabla A(x)\cdot x)u^{2}+((\alpha+\beta-2)B(x)-\nabla B(x)\cdot x)v^{2}\\
&\ \ \ \ +(\alpha+\beta-2)(u^{2}|\nabla u|^{2}+v^{2}|\nabla v|^{2})\bigg)\\
&>\frac{1}{2(N+\alpha+\beta)}\int_{\mathbb R^{N}}\bigg(t^{N}(\alpha+\beta)(|\nabla u|^{2}+ |\nabla v|^{2})+t^{N+2}(((\alpha+\beta-2)A(x)\\
&\ \ \ \ -\nabla A(x)\cdot x)u^{2}+((\alpha+\beta-2)B(x)-\nabla B(x)\cdot x)v^{2})\\
&\ \ \ \ +t^{N+\alpha+\beta}(\alpha+\beta-2)(u^{2}|\nabla u|^{2}+v^{2}|\nabla v|^{2})\bigg)\\
&=I(u_{t},v_{t})-\frac{1}{N+\alpha+\beta}G(u_{t} ,v_{t})\\
&\geq m, \endaligned$$
which is a contradiction.
\vskip4pt
{\bf Case 2:} $G(u,v)>0$. Set $\xi_{n}:=u_{n}-u,\gamma_{n}:=v_{n}-v$, by Lemma \ref{le32}, Lemma \ref{le33}, Br\'{e}zis-Lieb lemma \cite{bli}, $(A_{1})$ and $(B_{1})$, we may obtain
\begin{equation}\label{eq3.7}
G(u_{n},v_{n})\geq G(u,v)+G(\xi_{n},\gamma_{n})+o_{n}(1).
\end{equation}
Then $\limsup\limits_{n\rightarrow\infty}G(\xi_{n},\gamma_{n})<0$.  By Lemma \ref{le22}, there exists $t_{n}\in(0,1)$ such that $((\xi_{n})_{t_{n}}, (\gamma_{n})_{t_{n}})\in \mathcal{M}.$ Furthermore, one has that $\limsup\limits_{n\rightarrow\infty}t_{n}<1$, otherwise, along a subsequence, $t_{n}\rightarrow1$ and hence $G(\xi_{n},\gamma_{n})=G((\xi_{n})_{t_{n}}, (\gamma_{n})_{t_{n}})+o_{n}(1) = o_{n}(1)$, a contradiction. It follows from $(u_{n},v_{n})\in \mathcal{M} $, (\ref{eq3.1}), (\ref{eq3.7}), $(A_{2}) $ and $(B_{2})$  that
$$\aligned &\ m+o_{n}(1)\\
=&\ I( u_{n} ,v_{n} )-\frac{1}{N+\alpha+\beta}G(u_{n} ,v_{n})\\
=&\ \frac{1}{2(N+\alpha+\beta)}\int_{\mathbb R^{N}}\bigg((\alpha+\beta)(|\nabla u_{n}|^{2}+ |\nabla v_{n}|^{2})+((\alpha+\beta-2)A(x)\\
&\ \ \ \ -\nabla A(x)\cdot x)u_{n}^{2}+((\alpha+\beta-2)B(x)-\nabla B(x)\cdot x)v_{n}^{2}\\
&\ \ \ \ +(\alpha+\beta-2)(u_{n}^{2}|\nabla u_{n}|^{2}+v_{n}^{2}|\nabla v_{n}|^{2})\bigg)\\
\geq&\ \frac{\alpha+\beta}{2(N+\alpha+\beta)}\int_{\mathbb R^{N}}(|\nabla u |^{2}+|\nabla v |^{2}+|\nabla \xi_{n}|^{2}+|\nabla \gamma_{n}|^{2})\\
&\ +\frac{1}{2(N+\alpha+\beta)}\int_{\mathbb R^{N}}(((\alpha+\beta-2)A(x)-\nabla A(x)\cdot x)(u ^{2}\\
&\ +\xi_{n}^{2})+((\alpha+\beta-2)B(x)-\nabla B(x)\cdot x)(v ^{2}+\gamma_{n}^{2})) \\ &\ +\frac{\alpha+\beta-2}{2(N+\alpha+\beta)}\int_{\mathbb R^{N}}(u^{2}|\nabla u |^{2}+v ^{2}|\nabla v |^{2}+\xi_{n}^{2}|\nabla \xi_{n}|^{2}+\gamma_{n}^{2}|\nabla \gamma_{n}|^{2})\\
=&\ I( u ,v )-\frac{1}{N+\alpha+\beta}G(u,v)+I(\xi_{n},\gamma_{n})-\frac{1}{N+\alpha+\beta}G(\xi_{n}, \gamma_{n})\\
>&\ I( u ,v )-\frac{1}{N+\alpha+\beta}G(u ,v )+ I((\xi_{n}) _{t_n}, (\gamma_{n})_{t_n})\\
&\ -\frac{1}{N+\alpha+\beta}G((\xi_{n})_{t_n}, (\gamma_{n})_{t_n})\\
=&\ I((\xi_{n})_{t_n}, (\gamma_{n})_{t_n})+\frac{\alpha+\beta}{2(N+\alpha+\beta)}\int_{\mathbb R^{N}}(|\nabla u|^{2}+|\nabla v|^{2})\\
&\ +\frac{1}{2(N+\alpha+\beta)}\int_{\mathbb R^{N}}(((\alpha+\beta-2)A(x)-\nabla A(x)\cdot x)u^{2}\\
&\ +((\alpha+\beta-2)B(x)-\nabla B(x)\cdot x)v^{2})\\
&\  +\frac{\alpha+\beta-2}{2(N+\alpha+\beta)}\int_{\mathbb R^{N}}(u^{2}|\nabla u|^{2}+v^{2}|\nabla v|^{2})\\
\geq&\ m , \endaligned$$
which is also a contradiction.

Therefore, $(u,v)\in \mathcal{M},$ by using  Lebesgue dominated convergence theorem, Fatou's Lemma, $(A_{1}),(A_{2}),(B_{1})$, $(B_{2})$ and $(u_{n},v_{n})\in \mathcal{M} $, we may get
$$\aligned m&=I(u,v)-\frac{1}{N+\alpha+\beta}G(u,v)\\
&=\frac{1}{2(N+\alpha+\beta)}\int_{\mathbb R^{N}}\bigg((\alpha+\beta)(|\nabla u|^{2}+ |\nabla v|^{2})+((\alpha+\beta-2)A(x)\\
&\ \ \ \ -\nabla A(x)\cdot x)u^{2}+((\alpha+\beta-2)B(x)-\nabla B(x)\cdot x)v^{2}\\
&\ \ \ \ +(\alpha+\beta-2)(u^{2}|\nabla u|^{2}+v^{2}|\nabla v|^{2})\bigg)\\
&\leq \frac{1}{2(N+\alpha+\beta)}\liminf_{n\rightarrow+\infty} \int_{\mathbb R^{N}}\bigg((\alpha+\beta)(|\nabla u_{n}|^{2}+ |\nabla v_{n}|^{2})+((\alpha+\beta-2)A(x)\\
&\ \ \ \ -\nabla A(x)\cdot x)u_{n}^{2}+((\alpha+\beta-2)B(x)-\nabla B(x)\cdot x)v_{n}^{2}\\
&\ \ \ \ +(\alpha+\beta-2)(u_{n}^{2}|\nabla u_{n}|^{2}+v_{n}^{2}|\nabla v_{n}|^{2})\bigg)\\
&=\liminf_{n\rightarrow+\infty}\bigg(I( u_{n} ,v_{n} )-\frac{1}{N+\alpha+\beta}G(u_{n},v_{n})\bigg)\\
&=\liminf_{n\rightarrow+\infty}I( u_{n} ,v_{n} )\\
&=m,\endaligned$$ which implies that $\|(u_{n},v_{n})\|_{H}\rightarrow \|(u,v)\|_{H}$ and $I(u,v)=m$.
\vskip4pt
Having a minimum of  $I|_{\mathcal{M}}$, the fact that it is indeed a solution of  (\ref{eq11}), is based on a general idea used in  \cite[Lemma 2.5]{lww}.

\vskip12pt
\noindent{\bf Proof of Theorem \ref{th11}.}
\vskip4pt
Let $(\tilde{u},\tilde{v})\in \mathcal{M}$  be a minimizer of the functional $I|_{\mathcal{M}}$. Then from Lemma \ref{le22}, one has
$$I(\tilde{u},\tilde{v})=\inf_{(u,v)\in X }\max\limits_{t>0}I(u_{t},v_{t})=m.$$
Suppose by contradiction that $(\tilde{u},\tilde{v})$ is not a weak solution of (\ref{eq11}). Then, one can find $\phi_{1},\phi_{2}\in C_{0}^{\infty}(\mathbb R^{N})$  such that
$$\aligned &\ \langle I'(\tilde{u},\tilde{v}),(\phi_{1},\phi_{2})\rangle\\
=& \int_{\mathbb R^{N}}\bigg( \nabla \tilde{u}\nabla \phi_{1}+ \nabla \tilde{v}\nabla \phi_{2}+\nabla(\tilde{u}^{2})\nabla(\tilde{u}\phi_{1})+\nabla(\tilde{v}^{2})\nabla(\tilde{v}\phi_{2}) \\
&\ +A(x)\tilde{u}\phi_{1}+B(x)\tilde{v}\phi_{2}-\frac{2\alpha}{\alpha+\beta}|\tilde{u}|^{\alpha-2}\tilde{u}|\tilde{v}|^{\beta}\phi_{1}-\frac{2\beta}{\alpha+\beta}|\tilde{v}|^{\beta-2}\tilde{v}|\tilde{u}|^{\alpha}\phi_{2}\bigg)\\
<&-1.\endaligned$$
We choose small  $\varepsilon> 0$ such that
$$\langle I' (\tilde{u}_{t}+\sigma\phi_{1},\tilde{v}_{t}+\sigma\phi_{2}),(\phi_{1},\phi_{2}) \rangle\leq -\frac{1}{2},\ \ |t-1|,|\sigma|\leq\varepsilon, $$
and introduce a cut-off function $0\leq \zeta\leq1 $ satisfying $\zeta(t)=1 $ for $|t-1|\leq\frac{\varepsilon}{2}$ and $\zeta(t)=0 $ for $|t-1|\geq\varepsilon. $
For $ t\geq0$, we define
$$\gamma_{1}(t):=\left\{\aligned &\tilde{u}_{t}, &\hbox{if}\ |t-1|\geq\varepsilon,\\
&\tilde{u}_{t}+\varepsilon\zeta(t)\phi_{1}, &\hbox{if}\ |t-1|<\varepsilon,\endaligned\right. $$
$$\gamma_{2}(t):=\left\{\aligned &\tilde{v}_{t}, &\hbox{if}\ |t-1|\geq\varepsilon,\\
&\tilde{v}_{t}+\varepsilon\zeta(t)\phi_{2}, &\hbox{if}\ |t-1|<\varepsilon.\endaligned\right. $$
Then $ \gamma_{i}(t)$ is continuous curve in the metric space $ (H,d)$  and, eventually choosing a smaller $ \varepsilon$, we get that $d_{H}(\gamma_{i}(t),0)>0,$ for $|t-1|<\varepsilon$, $i=1,2$.

Next we claim that $\sup\limits_{t\geq0}I(\gamma_{1}(t),\gamma_{2}(t))<m.$  Indeed, if $|t-1|\geq\varepsilon $, then $$I(\gamma_{1}(t),\gamma_{2}(t))=I(\tilde{u}_{t},\tilde{v}_{t})<I(\tilde{u},\tilde{v})=m.$$
If $ |t-1|<\varepsilon$, by using the mean value theorem to the $C^{1}$  map $[0,\varepsilon]\ni \sigma\mapsto I(\tilde{u}_{t}+\sigma\zeta(t)\phi_{1},\tilde{v}_{t}+\sigma\zeta(t)\phi_{2})\in\mathbb R$, we find, for a suitable $\bar{\sigma}\in(0,\varepsilon)$,
$$\aligned &\ I(\tilde{u}_{t}+\sigma\zeta(t)\phi_{1},\tilde{v}_{t}+\sigma\zeta(t)\phi_{2})\\
=&\ I(\tilde{u}_{t},\tilde{v}_{t})+\langle I'(\tilde{u}_{t}+\bar{\sigma}\zeta(t)\phi_{1},\tilde{v}_{t}+\bar{\sigma}\zeta(t)\phi_{2}),(\zeta(t)\phi_{1},\zeta(t)\phi_{2})\rangle\\
\leq&\ I(\tilde{u}_{t},\tilde{v}_{t})-\frac{1}{2}\zeta(t)\\
<&\ m.\endaligned$$
To conclude observe that
$$G(\gamma_{1}(1-\varepsilon),\gamma_{2}(1-\varepsilon))>0,\ \ G(\gamma_{1}(1+\varepsilon),\gamma_{2}(1+\varepsilon))<0.$$
By using the continuity of the map $t\mapsto G(\gamma_{1}(t),\gamma_{2}(t))$  there exists $t_{0}\in(1-\varepsilon,1+\varepsilon) $  such that $G(\gamma_{1}(t_{0}),\gamma_{2}(t_{0}))=0$. Namely,
$$ (\gamma_{1}(t_{0}),\gamma_{2}(t_{0}))=(\tilde{u}_{t_{0}}+\varepsilon\zeta(t_{0})\phi_{1},\tilde{v}_{t_{0}}+\varepsilon\zeta(t_{0})\phi_{2})\in \mathcal{M},$$  and $I(\gamma_{1}(t_{0}),\gamma_{2}(t_{0}))<m$, this is a contradiction.

We claim $\tilde{u}\neq 0,\tilde{v}\neq0$. Indeed, if $\tilde{v}=0$, then the first equation of (\ref{eq11}) yields that $\tilde{u}=0$, then $ (\tilde{u},\tilde{v})=(0,0)$, this is impossible by the proof of {\bf Step 3} in Lemma \ref{le32}. Since any solution of (\ref{eq11}) belongs to  $\mathcal{M}$, the minimizer is a ground states. In addition, $(|\tilde{u}|,|\tilde{v}|)\in \mathcal{M}$ is also a minimizer, and hence a solution.  By the classical maximum principle, we get that $\tilde{u},\tilde{v}>0$.
\end{proof}

\bigskip

ACKNOWLEDGMENTS\\

The authors thank anonymous referees whose important comments helped them to improve their work. This work is supported by National Natural Science Foundation of China (Nos. 11871152; 11671085).\\

DATA AVAILABILITY\\

The data that support the findings of this study are available within the article.\\

\end{document}